\documentclass{article}
\usepackage[english]{babel} 		
\usepackage{fullpage}				
\usepackage[utf8]{inputenc}			
\usepackage{amsmath, amssymb,amsthm}
\usepackage{graphicx}				
\usepackage{hyperref}				
\usepackage{xcolor}
\usepackage{multirow}
\usepackage{tikz}
\usepackage{ytableau}
\usepackage{mathrsfs}
\usepackage{thmtools}
\usepackage{thm-restate}
\usepackage{varwidth}
\usepackage{comment}
\usepackage{mathtools}
\usepackage[mathscr]{euscript}
\usepackage{bbm}
\usepackage{multicol}
\usepackage{enumerate}
\usepackage{graphicx} 
\theoremstyle{definition}
\newtheorem{thm}{Theorem}[section]
\newtheorem{lem}[thm]{Lemma}
\newtheorem*{lem*}{Lemma}
\newtheorem*{thm*}{Theorem}
\newtheorem{prop}[thm]{Proposition}

\newtheorem{cor}[thm]{Corollary}

\newtheorem{defn}[thm]{Definition}
\newtheorem*{remark*}{Remark}
\newtheorem{remark}{Remark}
\newtheorem{example}{Example}

\newtheorem{cor/defn}[thm]{Corollary/Definition}

\DeclareMathOperator{\SYT}{\mathrm{SYT}}

\DeclareMathOperator{\sH}{\mathscr{H} }
\DeclareMathOperator{\sA}{\mathscr{A} }
\DeclareMathOperator{\sD}{\mathscr{D} }
\DeclareMathOperator{\sP}{\mathscr{P} }
\DeclareMathOperator{\Ind}{Ind}
\DeclareMathOperator{\Res}{Res}

\DeclareMathOperator{\MOD}{\mathrm{Mod}}

\DeclareMathOperator{\bB}{\mathbb{B}_{q,t}}
\DeclareMathOperator{\fL}{\mathfrak{L}}
\DeclareMathOperator{\pol}{\mathrm{pol}}
\DeclareMathOperator{\ext}{\mathrm{ext}}
\DeclareMathOperator{\Exp}{\mathrm{Exp}}

\usepackage[
    backend=bibtex,
    maxnames=5,
    style=alphabetic
  ]{biblatex}
\usepackage{graphicx} 

\title{Double Dyck Path Algebra Representations From DAHA}
\author{Milo Bechtloff Weising}
\date{\today}

\begin{document}

\maketitle

\abstract{The double Dyck path algebra $\mathbb{A}_{q,t}$ was introduced by Carlsson-Mellit in their proof of the Shuffle Theorem. A variant of this algebra, $\mathbb{B}_{q,t}$, was introduced by Carlsson-Gorsky-Mellit in their study of the parabolic flag Hilbert schemes of points in $\mathbb{C}^2$ showing that $\mathbb{B}_{q,t}$ acts naturally on the equivariant $K$-theory of these spaces. The algebraic relations defining $\mathbb{B}_{q,t}$ appear superficially similar to those of the positive double affine Hecke algebras (DAHA) in type $GL$, $\mathscr{D}_n^{+}$, introduced by Cherednik. In this paper we provide a general method for constructing $\mathbb{B}_{q,t}$ representations from DAHA representations. In particular, every $\mathscr{D}_n^{+}$ module yields a representation of a subalgebra $\mathbb{B}_{q,t}^{(n)}$ of $\mathbb{B}_{q,t}$ and special families of compatible DAHA representations give representations of $\mathbb{B}_{q,t}$. These constructions are functorial. Lastly, we will construct a large family of $\mathbb{B}_{q,t}$ representations indexed by partitions using this method related to the Murnaghan-type representations of the positive elliptic Hall algebra introduced previously by the author.}

\tableofcontents

\section{Introduction}

The algebra $\mathbb{B}_{q,t}$ was introduced by Carlsson-Gorsky-Mellit \cite{GCM_2017} as an algebra which has a natural geometric action on the equivariant $K$-theory of the parabolic flag Hilbert schemes of points in $\mathbb{C}^2.$ This work built on the prior work of Schiffmann-Vasserot \cite{SV} who constructed a geometric action of the elliptic Hall algebra $\mathcal{E}$ on the equivariant $K$-theory of the Hilbert schemes of points in $\mathbb{C}^2.$ These construction are part of a larger story in Macdonald theory of relating geometric properties of the Hilbert schemes of points in $\mathbb{C}^2$ to the algebraic combinatorics underlying the modified Macdonald symmetric functions $\widetilde{H}_{\mu}$ and of the Macdonald operator $\Delta$ (which acts on the space of symmetric functions $\Lambda$). Importantly, $\mathbb{B}_{q,t}$ is intimately related to the \textit{double Dyck path algebra} $\mathbb{A}_{q,t}$ introduced by Carlsson-Mellit in their proof of the Shuffle Theorem \cite{CM_2015} regarding the Frobenius character of the space of diagonal coinvariants and the combinatorics of Dyck paths. 

The quiver path algebra $\mathbb{B}_{q,t}$ has relations very similar to the \textit{positive double affine Hecke algebras} (DAHA) in type $GL$, $\sD_n^{+}$, introduced by Cherednik \cite{C_2001}. In fact, $\mathbb{B}_{q,t}$ contains many copies of \textit{affine Hecke algebras} of type $GL.$ However, there is no direct algebraic relation (no algebra homomorphisms) between $\mathbb{B}_{q,t}$ (nor $\mathbb{A}_{q,t}$) and DAHAs. Nevertheless, there are approaches to more indirectly relate these algebras. Ion-Wu \cite{Ion_2022} defined an algebra called the \textit{stable-limit DAHA} along with a polynomial representation on the space of \textit{almost symmetric functions} $\sP_{as}^{+}$ which, in a sense, globalizes the polynomial representation of $\mathbb{A}_{q,t}$ (and as we will see later $\mathbb{B}_{q,t}$). They used a stable-limit procedure to define this representation from the polynomial representations of the finite rank DAHAs $\sD_n^{+}$. This representation of the stable-limit DAHA is much larger than the polynomial representation of $\mathbb{B}_{q,t}$ but the limit Cherednik operators of Ion-Wu, in a sense, behave better on a certain subspace of $\sP_{as}^{+}$ given by the following direct sum:
$$\bigoplus_{k \geq 0} x_1\cdots x_k\mathbb{Q}(q,t)[x_1,\ldots,x_k] \otimes \Lambda .$$ This subspace aligns with the polynomial representation of $\mathbb{B}_{q,t}.$

Motivated by the construction of Ion-Wu we will in this paper develop a method for constructing modules for $\mathbb{B}_{q,t}$ directly from the representation theory of DAHA in type $GL.$ We will use a stable-limit construction similar to Ion-Wu but will not require any additional non-archimedean topological considerations as they did. First, we will show (Proposition \ref{finite rank main construction}) that given any $\sD_n^{+}$ module $V$ we may construct an action of the subalgebra $\mathbb{B}_{q,t}^{(n)}$ on the space $L_{\bullet}(V)$ defined by 
$$L_{\bullet}(V)= \bigoplus_{0\leq k \leq n} L_{k}(V) := \bigoplus_{0\leq k \leq n} X_1\cdots X_k \epsilon_k(V).$$
Here $\epsilon_k$ are the \textit{partial trivial idempotents} of the finite Hecke algebra. Each space may be considered as a module for the partially symmetrized positive DAHA, $\epsilon_k\sD_n^{+}\epsilon_k.$ It will be immediate to show (Theorem \ref{finite rank functorial}) that the map $V \rightarrow L_{\bullet}(V)$ is indeed a functor. We show (Proposition \ref{quotient map prop for poly}) that in the case of the polynomial representations $V_{\pol}^{(n)}$ of DAHA that $L_{\bullet}(V_{\pol}^{(n)})$ is a $\mathbb{B}_{q,t}^{(n)}$-module quotient of the restriction of the polynomial representation of $\mathbb{B}_{q,t}$ to $\mathbb{B}_{q,t}^{(n)}$.

Afterwards, we will use stable-limits of the representations $L_{\bullet}(V)$ of $\mathbb{B}_{q,t}^{(n)}$ to build representations of $\mathbb{B}_{q,t}.$ This construction will require the input of an infinite family of representations of DAHAs, $(V^{(n)})_{n \geq n_0}$, along with some connecting maps, $\Pi^{(n)}: V^{(n+1)} \rightarrow V^{(n)}$, satisfying some special assumptions. Most interestingly, we require that the following relations holds:
$$\Pi^{(n)}\pi^{(n+1)}T_n = \pi^{(n)}\Pi^{(n)}.$$ This is the same relation used by Ion-Wu in their construction of the limit Cherednik operators and is related to certain natural embeddings of the extended affine symmetric groups $\widetilde{\mathfrak{S}}_n \hookrightarrow \widetilde{\mathfrak{S}}_{n+1}.$ We call such families $C = \left( (V^{(n)})_{n \geq n_1}, (\Pi^{(n)})_{n\geq n_1}\right)$ \textit{compatible} and construct spaces $\fL_k(C)$ given by
$$\fL_k(C):= \lim_{\leftarrow} L_k(V^{(n)}).$$ These are the stable-limits of the spaces $L_k(V^{(n)})$ with respect to the maps $\Pi^{(n)}.$ Finally, we package together these spaces to form $\fL_{\bullet}(C)$ given as 
$$\fL_{\bullet}(C):= \bigoplus_{k \geq 0} \fL_k(C)$$
which may be also thought of as the stable-limit of the $\mathbb{B}_{q,t}^{(n)}$ modules $L_{\bullet}(V^{(n)}).$ We show (Theorem \ref{main theorem}) that there is a natural action of $\mathbb{B}_{q,t}$ on $\fL_{\bullet}(C)$ determined by the $\mathbb{B}_{q,t}^{(n)}$ module structures on $L_{\bullet}(V^{(n)}).$ This construction is also functorial.

Lastly, we will use our construction of the functor $C \rightarrow \fL_{\bullet}(C)$ to define (Theorem \ref{Murnaghan-type reps}) a large family of $\mathbb{B}_{q,t}$ modules, $\fL_{\bullet}(\Ind(C_{\lambda}))$, indexed by partitions $\lambda.$ These representations in a sense extend the  Murnaghan-type representations of the positive elliptic Hall algebra previously defined by the author \cite{weising2023murnaghantype}. As such we call these the Murnaghan-type representations of $\mathbb{B}_{q,t}.$ For $\lambda = \emptyset$, $\fL_{\bullet}(\Ind(C_{\emptyset}))$ recovers the polynomial representation of $\mathbb{B}_{q,t}.$

In a recent paper Gonz\'{a}lez-Gorsky-Simental \cite{gonzález2023calibrated} defined an extension $\mathbb{B}_{q,t}^{\ext}$ of $\mathbb{B}_{q,t}$ containing certain additional $\Delta$-operators as well as a class of representations of $\mathbb{B}_{q,t}^{\ext}$ called \textit{calibrated} with special properties. Further, they construct a large class of \textit{calibrated} $\mathbb{B}_{q,t}^{\ext}$ representations built from certain posets with exceptional properties. The author conjectures that the Murnaghan-type representations of $\mathbb{B}_{q,t}$, $\fL_{\bullet}(\Ind(C_{\lambda}))$, have extended actions by $\mathbb{B}_{q,t}^{\ext}$ which are calibrated. More generally, there should be a special set of conditions on a compatible sequence $C$ which guarantees that $\fL_{\bullet}(C)$ has an extended action by $\mathbb{B}_{q,t}^{\ext}$ which is calibrated. The author will return to these questions in the future. 

\subsection{Acknowledgements}
The author would like to thank their advisor Monica Vazirani for all of her help and guidance. The author would also to thank Nicolle Gonz\'{a}lez, Eugene Gorsky, and Jos\'{e} Simental for helpful conversations with the author regarding $\mathbb{B}_{q,t}^{\ext}$ and calibrated representations. The author would also like to thank Daniel Orr for fruitful conversations regarding the polynomial representation of $\mathbb{B}_{q,t}.$
\section{Notation and Conventions}

\subsection{Stable-Limits}

In this paper \textit{graded} will mean $\mathbb{Z}_{\geq 0}$-graded for both vector spaces and rings. We fix two indeterminants $q,t$ which are algebraically independent over $\mathbb{Q}$ and consider all vector spaces to be over the base field $\mathbb{Q}(q,t).$ We will write $\deg(v), \deg(r)$ for the degree of either a homogeneous vector $v$ in a graded vector space or a homogeneous element $r$ of a graded ring. For a graded vector space $V$ we will write $V(d)$ for the degree $d \geq 0$ homogeneous component of $V.$ If $R$ is a graded ring then we will write $R-\MOD$ for the category of consisting of graded left $R$ modules as the objects and with degree-preserving homomorphisms (homogeneous maps) as the morphisms.

We now review some formalities regarding stable-limits of spaces and modules. 

\begin{defn}\label{stable-limit def}
    Let $(V^{(n)})_{n \geq n_0}$ be a sequence of graded vector spaces and suppose $(\Pi^{(n)}: V^{(n+1)} \rightarrow V^{(n)})_{n \geq n_0}$ is a family of degree preserving maps. The \textbf{\textit{stable-limit}} of the spaces $(V^{(n)})_{n \geq n_0}$ with respect to the maps $(\Pi^{(n)})_{n \geq n_0}$ is the graded vector space $\widetilde{V} := \lim_{\leftarrow}V^{(n)}$ constructed as follows:
    For each $d \geq 0$ we define 
    $$\widetilde{V}(d):= \{(v_n)_{n \geq n_0} \in \prod_{n\geq n_0}V^{(n)}(d) ~|~ \Pi^{(n)}(v_{n+1}) = v_n\}$$ and set 
    $$\widetilde{V}:= \bigoplus_{d \geq 0} \widetilde{V}(d).$$
\end{defn}

\begin{lem}\label{stable-limit of modules}
    Let $(R^{(n)})_{n \geq n_0}$ be a sequence of graded rings with injective graded ring homomorphisms $(\iota^{(n)}: R^{(n)} \rightarrow R^{(n+1)})_{n \geq n_0}$. We will identify $R^{(n)}$ with its image $\iota_n(R^{(n)}) \subset R^{(n+1)}.$ We write $\widetilde{R} = \lim_{\rightarrow} R^{(n)}$ for the direct limit of the rings $R^{(n)}.$ Suppose $(V^{(n)})_{n \geq n_0}$ is a sequence of graded vector spaces with each $V^{(n)}$ a graded $R^{(n)}$ module and $(\Pi^{(n)}:V^{(n+1)} \rightarrow V^{(n)} )_{n \geq n_0}$ a sequence of degree-preserving maps with each $\Pi^{(n)}$ a graded $R^{(n)}$ module homomorphism. Then the following defines a graded $\widetilde{R}$ module structure on $\widetilde{V}:= \lim_{\leftarrow} V^{(n)}$:
    For $r \in \widetilde{R}$ and $v \in \widetilde{V}$ with $r \in R_{N}$ and $v = (v_{n})_{n \geq n_0}$, define $r(v) \in \widetilde{V}$ by 
    $$r(v) = \left( \Pi^{(n_0)}\cdots \Pi^{(N-1)}(r(v_{N}))  ,\ldots, \Pi^{(N-1)}(r(v_{N})), r(v_{N}), r(v_{N+1}), r(v_{N+2}),\ldots \right).$$
\end{lem}

\begin{remark}\label{functoriality of stable-limit of modules}
It is a straightforward exercise to check that the action defined above actually yields a graded $\widetilde{R}$ module structure on $\lim_{\leftarrow} V^{(n)}.$ We leave this to the reader.  We call $\lim_{\leftarrow} V^{(n)}$ the \textbf{\textit{stable-limit module}} corresponding to the sequence $(V^{(n)})_{n \geq n_0}$ and the maps $(\Pi^{(n)})_{n \geq n_0}.$ Notice that this construction is \textit{functorial}. Suppose $(W^{(n)})_{n \geq n_0}$ is another sequence of graded vector spaces with each $W^{(n)}$ a graded $R^{(n)}$ module and $(\Psi^{(n)}:W^{(n+1)} \rightarrow W^{(n)} )_{n \geq n_0}$ a sequence of degree-preserving maps with each $\Psi^{(n)}$ a graded $R^{(n)}$ module homomorphism and $\phi = (\phi^{(n)})_{n \geq n_0}$ is a family of graded $R^{(n)}$ module maps $\phi^{(n)}: V^{(n)} \rightarrow W^{(n)}$ such that for all $n \geq n_0$
    $$ \phi^{(n)}\Pi^{(n)} = \Psi^{(n)}\phi^{(n+1)}.$$ 
Then $\phi$ determines a graded $\widetilde{R}$ module homomorphism $\widetilde{\phi}: \lim_{\leftarrow} V^{(n)} \rightarrow \lim_{\leftarrow} W^{(n)} $ given by 
$$\widetilde{\phi}(v): = (\phi^{(n)}(v_n))_{n\geq n_0}.$$
\end{remark}

\begin{remark}
    The stable-limit spaces $\widetilde{V} = \lim_{\leftarrow} V^{(n)}$ may be zero even if each $V^{(n)}$ is nonzero. However, if each $V^{(n)}$ is nonzero and the maps $\Pi^{(n)}$ are surjective then $\widetilde{V}$ is nonzero.
\end{remark}

\subsection{Positive Double Affine Hecke Algebras}

\begin{defn} 
Define the \textbf{\textit{positive double affine Hecke algebra}} $\sD_n^{+}$ to be the $\mathbb{Q}(q,t)$-algebra generated by $T_1,\ldots,T_{n-1}$, $X_1,\ldots,X_{n}$, and $Y_1,\ldots, Y_n$ with the following relations:

\begin{multicols}{2}
    \item
    $(T_i -1)(T_i +q) = 0$,
    \item  $T_iT_{i+1}T_i = T_{i+1}T_iT_{i+1}$,
    \item  $T_iT_j = T_jT_i$, $|i-j|>1$,
    \item  $T_i^{-1}X_iT_i^{-1} = q^{-1}X_{i+1}$,
    \item $T_iX_j = X_jT_i$, $j \notin \{i,i+1\}$,
    \item $X_iX_j = X_jX_i$,
    \item $T_iY_iT_i = qY_{i+1}$,
    \item $T_iY_j = Y_jT_i$, $j\notin \{i,i+1\}$,
    \item $Y_iY_j = Y_jY_i$,
    \item $Y_1T_1X_1 = X_2Y_1T_1$,
    \item $Y_1X_1\cdots X_n = tX_1\cdots X_nY_1.$
\end{multicols}
We will write $\sA_n^{X}$ for the subalgebra of $\sD_{n}^{+}$ generated by $T_1,\ldots, T_{n-1},X_1,\ldots, X_n$ and $\sA_n$ for the subalgebra generated by $T_1,\ldots, T_{n-1}, Y_1\ldots, Y_n.$ Further, we write $\sH_n$ for the finite Hecke algebra which is the subalgebra of $\sD_n^{+}$ generated by $T_1,\ldots, T_{n-1}.$

We define the special elements $\pi:= q^{-n}Y_1T_1\cdots T_{n-1}$ and $\widetilde{\pi}:= X_1T_1^{-1}\cdots T_{n-1}^{-1}.$ Further, we define for $0\leq k \leq n$ the element $\epsilon_k$ by 
$$\epsilon_k: = \frac{1}{[n-k]_q!}\sum_{\sigma \in \mathfrak{S}_{(1^k,n-k)}} q^{{n -k \choose 2} -\ell(\sigma)}T_{\sigma}$$
where $T_{\sigma}:= T_{i_1}\cdots T_{i_r}$ whenever $\sigma = s_{i_1}\cdots s_{i_r}$ is a reduced expression for $\sigma$, $\mathfrak{S}_{(1^k,n-k)}$ is the Young subgroup of the symmetric group $\mathfrak{S}_{n}$ corresponding to the composition $(1^k,n-k)$, and $[m]_q!:= \prod_{i=1}^{m}\left(\frac{1-q^i}{1-q}\right).$ 

We will consider $\sD_n^{+}$ as a graded algebra with
\begin{itemize}
    \item $\deg(T_i) = \deg(Y_i) = 0$
    \item $\deg(X_i) = 1.$
\end{itemize}

\end{defn}

\begin{remark}\label{additional relations remark}
    It is straightforward to check the following additional relations which are all standard in DAHA theory. We will require all of these relations later in this paper. For the element $\pi$ we have:
    \begin{itemize}
        \item $\pi X_i = X_{i+1}\pi$ for $1\leq i \leq n-1$
        \item $\pi T_i = T_{i+1}\pi$ for $1\leq i \leq n-1$
        \item $\pi^{2}T_{n-1} = T_1\pi^{2}.$
    \end{itemize}

    The elements $\epsilon_k$ are the \textit{partial trivial} idempotents. They satisfy the relations:
    \begin{itemize}
        \item $\epsilon_k^{2} = \epsilon_k$
        \item $\epsilon_k T_i = T_i \epsilon_k = \epsilon_k$ for $k+1 \leq i \leq n-1$
        \item $T_i\epsilon_k = \epsilon_kT_i$ for $1 \leq i \leq k-1$
        \item $\epsilon_k = \frac{1}{[n-k]_q!} \sum_{\sigma \in \mathfrak{S}_{(1^k,n-k)}} q^{\ell(\sigma)}T_{\sigma}^{-1}$
        \item $\epsilon_k = \left(\frac{1+qT_{k+1}^{-1}+\ldots + q^{n-k-1}T_{n-1}^{-1}\cdots T_{k+1}^{-1} }{1+q+\ldots + q^{n-k-1}} \right)\epsilon_{k+1} $
        \item $\epsilon_k \epsilon_{\ell} = \epsilon_{\min(k,\ell)}.$
    \end{itemize}
    We have that the important element $\widetilde{\pi}:= X_1T_1^{-1}\cdots T_{n-1}^{-1}$ satisfies the relations:
    \begin{itemize}
        \item $\widetilde{\pi}Y_i = Y_{i+1}\widetilde{\pi}$ for $1\leq i \leq n-1$
        \item $\widetilde{\pi}tY_n = Y_{1}\widetilde{\pi}$
        \item $\widetilde{\pi}T_i = T_{i+1}\widetilde{\pi}$ for $1 
        \leq i \leq n-2$
        \item $\widetilde{\pi}^2T_{n-1} = T_{1}\widetilde{\pi}^2.$
    \end{itemize}

    Note that the last two of these relations only depend of the structure of the subalgebra $\sA_n^{X}$ of $\sD_n^{+}$ and thus hold more generally for all $ 2 \leq k \leq n$:
    \begin{itemize}
        \item $(X_1T_1^{-1}\cdots T_{k-1}^{-1})T_i = T_{i+1}(X_1T_1^{-1}\cdots T_{k-1}^{-1})$ for $1 
        \leq i \leq k-2$
        \item $(X_1T_1^{-1}\cdots T_{k-1}^{-1})^2T_{k-1} = T_{1}(X_1T_1^{-1}\cdots T_{k-1}^{-1})^2.$
    \end{itemize}

    Lastly, we have the following expansion of the Jucys-Murphy elements of the finite Hecke algebra into the standard $T_{\sigma}$ basis:
    $$q^{n-k}T_k^{-1}\cdots T_{n-1}^{-1}T_{n-1}^{-1}\cdots T_k^{-1} = 1 +(q-1)(T_k^{-1}+ qT_{k+1}^{-1}T_{k}^{-1}+ \ldots + q^{n-k-1}T_{n-1}^{-1}\cdots T_k^{-1}).$$
\end{remark}

\begin{remark}
    When considering $\sD_n^{+}$ and $\sD_m^{+}$ for $n \neq m$ we will write $Y_i^{(n)}, \pi^{(n)}, \epsilon_k^{(n)}$ to help distinguish between the distinct copies of $Y_i, \pi, \epsilon_k$ in $\sD_n^{+}$ and $\sD_m^{+}.$ However, since $\sA_{n}^{X}$ naturally includes into $\sA_{m}^{X}$ for any $n \leq m$ we will identify the copies of $X_i$ and $T_i$ in $\sA_{n}^{X}$ and $\sA_{m}^{X}.$
\end{remark}

\begin{defn}\cite{C_2001}\label{polynomial rep of daha}
    The polynomial representation $V_{\pol}^{(n)}$ of $\sD_n^{+}$ consists of the space $\mathbb{Q}(q,t)[x_1,\ldots, x_n]$ along with the following operators:
    \begin{itemize}
        \item $T_i(f) = s_i(f) + (q-1)x_i\frac{f-s_i(f)}{x_i-x_{i+1}}$
        \item $\pi(f(x_1,\ldots,x_n)) = f(x_2,\ldots, x_n, tx_1)$
        \item $X_i(f) = x_if.$
    \end{itemize}
\end{defn}

It is a standard exercise in DAHA theory to verify that $V_{\pol}^{(n)}$ is a $\sD_n^{+}$ module as defined above (see \cite{C_2001}).

\subsection{The $\mathbb{B}_{q,t}$ Algebra}

\begin{defn}\label{B_qt relations}\cite{GCM_2017}
    The algebra $\mathbb{B}_{q,t}$ is generated by a collection of orthogonal idempotents labelled by $\mathbb{Z}_{\geq0}$, generators $d_{+}$, $d_{-}$, $T_i$, and $z_i$ modulo relations:
\begin{multicols}{2}
     \item $(T_i-1)(T_i+q) = 0$
    \item $T_iT_{i+1}T_i = T_{i+1}T_i T_{i+1}$
    \item $T_iT_j = T_j T_i$ if $|i-j|> 1$
    \item $T_i^{-1}z_{i+1}T_{i}^{-1} = q^{-1}z_i$ for $1 \leq i \leq k-1$
    \item $z_iT_j = T_j z_i$ if $i \notin \{j,j+1\}$
    \item $z_iz_j = z_jz_i$ for $1\leq i,j \leq k$
    \item $d_{-}^{2}T_{k-1} = d_{-}^{2}$ for $ k \geq 2$
    \item $d_{-}T_i = T_i d_{-}$ for $1\leq i \leq k-2$
    \item $T_1d_{+}^{2}= d_{+}^{2}$
    \item $d_{+}T_i = T_{i+1} d_{+}$ for $1 \leq i \leq k-1$
    \item $q\varphi d_{-} = d_{-} \varphi T_{k-1}$ for $k \geq 2$
    \item $T_1 \varphi d_{+} = q d_{+} \varphi$ for $k \geq 1$
    \item $z_id_{-} = d_{-}z_i$
    \item $d_{+}z_i = z_{i+1}d_{+}$
    \item $z_1(qd_{+}d_{-} - d_{-}d_{+}) = qt(d_{+}d_{-} - d_{-}d_{+})z_k$ for $ k \geq 1$
\end{multicols}
where $\varphi:= \frac{1}{q-1}[d_{+},d_{-}].$

We will consider $\mathbb{B}_{q,t}$ as a graded algebra with grading determined by 
\begin{itemize}
    \item $\deg(T_i) = \deg(z_i) = \deg(d_{-}) = 0$
    \item $\deg(d_{+}) = 1.$
\end{itemize}

For $n \geq 0$ define $\mathbb{B}_{q,t}^{(n)}$ to be the subalgebra of $\mathbb{B}_{q,t}$ given by only considering $T_i,z_i,d_{-},d_{+}$ between the idempotents labelled by $\{0,\ldots, n\}.$ 
\end{defn}

\begin{remark}
    The graded algebras $\mathbb{B}_{q,t}^{(n)}$ naturally form a directed system with 
    $$\bB = \lim_{\rightarrow} \mathbb{B}_{q,t}^{(n)}.$$
\end{remark}

\begin{defn}\cite{GCM_2017}
    Let $V^{\pol}_{\bullet} = \bigoplus_{k \geq 0} V^{\pol}_{k} := \bigoplus_{k \geq 0} \mathbb{Q}(q,t)[y_1,\ldots, y_k] \otimes \Lambda$ where $\Lambda = \Lambda[X]$ is the ring of \textbf{\textit{symmetric functions}} (over $\mathbb{Q}(q,t)$) in a variable set $X = x_1+ x_2+\ldots.$ Define an action on $V^{\pol}_{\bullet}$ by the following operators given for $F \in V^{\pol}_{k}$ by
\begin{itemize}
    \item $T_iF:= \frac{(q-1)y_{i}F + (y_{i+1}-qy_i)s_i(F)}{y_{i+1}-y_i}$ for $1 \leq i \leq k-1$
    \item $d_{+}F := T_1\cdots T_k F[X+(q-1)y_{k+1}]$
    \item $d_{-}F := -F[X-(q-1)y_k]\Exp[-y_k^{-1}X]|_{y_k^{-1}}$
    \item $z_kF := T_{k-1}\cdots T_1 F[X +(q-1)ty_1 - (q-1)u, y_2,\ldots, y_k, u] \Exp[u^{-1}ty_1 - u^{-1}X]|_{u^{0}}$
    \item $z_{i}:= q^{-1}T_i^{-1}z_{i+1}T_i^{-1}$ for $1 \leq i \leq k-1$
\end{itemize}
where $\Exp[X]:= \sum_{n \geq 0} h_n[X]$ is the \textit{plethystic exponential}, $|_{y_k^{-1}}$ represents taking the coefficient of $y_k^{-1},$ and $s_i$ swaps the variables $y_i,y_{i+1}$. Here we are using \textit{plethystic} notation. This representation $V^{\pol}_{\bullet}$ of $\mathbb{B}_{q,t}$ is called the \textbf{\textit{polynomial representation}}.

\begin{remark}
    Carlsson-Gorsky-Mellit also construct an action of $\mathbb{B}_{q,t}$ on the larger space $W^{\pol}_{\bullet}:= \bigoplus_{k=0}^{\infty} (y_1\cdots y_{k})^{-1}V^{\pol}_{k}.$ The space $V^{\pol}_{\bullet}$ is isomorphic to the equivariant $K$-theory of the parabolic flag Hilbert schemes of points in $\mathbb{C}^2$ and the larger space $W^{\pol}_{\bullet}$ is defined in order to relate the original $\mathbb{A}_{q,t}$ polynomial representation as defined by Carlsson-Mellit \cite{CM_2015} to the $\mathbb{A}_{q,t}$ polynomial representation constructed in \cite{GCM_2017}. We will use the space $W^{\pol}_{\bullet}$ briefly to relate the $\mathbb{B}_{q,t}$ action on $V^{\pol}_{\bullet}$ to the work of Ion-Wu.
\end{remark}
    
\end{defn}

\section{Main Construction}
\subsection{$\mathbb{B}_{q,t}^{(n)}$ Modules From $\sD_n^{+}$}

In this section we will take any graded $\sD_n^{+}$ module $V$ and construct a corresponding graded $\mathbb{B}_{q,t}^{(n)}$ module $L_{\bullet}(V).$ To do this we will first define the spaces which constitute $L_{\bullet}(V).$

\begin{defn}
   For any graded $\sD_n^{+}$ module $V$ and $0 \leq k \leq n$ define the space $L_k = L_k(V)$ as $L_k:= X_1\cdots X_k \epsilon_k(V).$ We let $L_{\bullet} = L_{\bullet}(V)$ denote the space $L_{\bullet}:= \bigoplus_{0 \leq k \leq n} L_k.$
\end{defn}

For the remainder of this section let $V$ be a graded module for $\sD_n^{+}.$ We are going to now construct operators $T_i,z_i,d_{+},d_{-}$ on $L_{\bullet}$ which we will show generate a representation of $\mathbb{B}_{q,t}^{(n)}.$ 

\begin{defn}\label{operators construction}
    Define the operators
    \begin{itemize}
        \item $T_i: L_{k} \rightarrow L_{k}$ for $1 \leq i \leq k-1$
        \item $z_i: L_{k} \rightarrow L_{k}$ for $1 \leq i \leq k$
        \item $d_{+}:L_{k} \rightarrow L_{k+1}$ for $0 \leq k \leq n-1$
        \item $d_{-}:L_{k} \rightarrow L_{k-1}$ for $1\leq k \leq n$
    \end{itemize}
as follows:

\begin{itemize}
    \item $T_i(v)$ is defined by the action of $T_i$ on $V$
    \item $z_i(v):= (qt)^{-1}Y_i(v)$ as defined by the action of $Y_i$ on $V$
    \item $d_{+}(v):= q^{k}X_1T_1^{-1}\cdots T_k^{-1}v$
    \item $d_{-}(v):=  (q-1)(1+qT_{k}^{-1}+\ldots + q^{n-k}T_{n-1}^{-1}\cdots T_{k}^{-1})(v).$
\end{itemize}
\end{defn}

It is not immediately obvious that these operators are well defined i.e. that their ranges are correctly specified above. We show this now.

\begin{lem}\label{well-definedness of operators}
    If $v\in L_k$ then $T_i(v), z_j(v) \in L_{k}$ for all $1 \leq i\leq k-1$ and $1\leq j\leq k.$ If $ k \leq n-1$ then $d_{+}(v) \in L_{k+1}$ and if $1\leq k$ then $d_{-}(v) \in L_{k-1}.$ 
\end{lem}

\begin{proof}
    Let $v \in L_k$ say, $v = X_1\cdots X_k\epsilon_k(w)$ for $w \in V.$
    First we have, 
    \begin{align*}
        T_i(v) &= T_iX_1\cdots X_k\epsilon_k(w)\\
        &= X_1\cdots X_k\epsilon_k(T_iw) \in L_k.\\
    \end{align*}
    Next we have 
    \begin{align*}
        z_1(v) &= (qt)^{-1}Y_1X_1\cdots X_k\epsilon_k(w) \\
        &= (qt)^{-1}q^{n}\pi T_{n-1}^{-1}\cdots T_{1}^{-1}X_1\cdots X_k\epsilon_k(w)\\
        &= (qt)^{-1}q^n \pi T_{n-1}^{-1}\cdots T_{k}^{-1}X_1\cdots X_k T_{k-1}^{-1}\cdots T_{1}^{-1} \epsilon_k(w)\\
        &= (qt)^{-1}q^n \pi X_1 \cdots X_{k-1} T_{n-1}^{-1}\cdots T_k^{-1} X_k T_{k-1}^{-1}\cdots T_1^{-1} \epsilon_k(w)\\
        &= (qt)^{-1}q^n X_2 \cdots X_k \pi T_{n-1}^{-1}\cdots T_k^{-1} X_k T_{k-1}^{-1}\cdots T_1^{-1}\epsilon_k(w)\\
        &= (qt)^{-1}q^n X_2\cdots X_k tq^{-(n-k)}X_1\pi T_{n-1}\cdots T_kT_{k-1}^{-1}\cdots T_1^{-1}\epsilon_k(w)\\
        &= (qt)^{-1}tq^kX_1\cdots X_k \pi T_{n-1}\cdots T_kT_{k-1}^{-1}\cdots T_1^{-1}\epsilon_k(w)\\
        &= (qt)^{-1}X_1\cdots X_k \pi T_{n-1}\cdots T_k \epsilon_k(tq^kT_{k-1}^{-1}\cdots T_1^{-1}w).\\
    \end{align*}

    Now for all $k < i \leq n-1$ we have 
    \begin{align*}
        T_i\pi T_{n-1}\cdots T_k &= \pi T_{i-1}T_{n-1}\cdots T_k \\
        &= \pi T_{n-1}\cdots T_{i+1} T_{i-1}T_iT_{i-1} T_{i-2} \cdots T_k\\
        &= \pi T_{n-1}\cdots T_{i+1} T_{i}T_{i-1}T_{i} T_{i-2} \cdots T_k\\
        &= \pi T_{n-1}\cdots T_k T_i.\\
    \end{align*}
    Therefore, we have 
    \begin{align*}
        &X_1\cdots X_k \pi T_{n-1}\cdots T_k \epsilon_k(tq^kT_{k-1}^{-1}\cdots T_1^{-1}w)\\
        &= X_1\cdots X_k \epsilon_k(tq^k \pi T_{n-1}\cdots T_kT_{k-1}^{-1}\cdots T_1^{-1}w)\\
    \end{align*}
    which is clearly in $L_k.$ 

    Now for any $1< i \leq k$ since 
    $Y_i = q^{-1}T_{i-1}Y_{i-1}T_{i-1}$ we see that 
    $$Y_i = q^{-i+1}T_{i-1}\cdots T_1Y_1T_1\cdots T_{i-1}$$ and so 
    $$z_i = (qt)^{-1}q^{-i+1}T_{i-1}\cdots T_1Y_1T_1\cdots T_{i-1}.$$ Since $T_1\cdots T_{i-1}v \in L_k$ we see that $Y_1(T_1\cdots T_{i-1}v) \in L_k$ as well and so
    $$z_i(v) = (qt)^{-1}q^{-i+1}T_{i-1}\cdots T_1Y_1T_1\cdots T_{i-1}(v) = (qt)^{-1}q^{-i+1}T_{i-1}\cdots T_1Y_1(T_1\cdots T_{i-1}v) \in L_k.$$

    We now look at $d_{+}.$ We find that
    \begin{align*}
        &d_{+}(v) = \\
        &= q^{k}X_1T_1^{-1}\cdots T_k^{-1}(v)\\
        &= q^{k}X_1T_1^{-1}\cdots T_k^{-1}(X_1\cdots X_k\epsilon_k(w))\\  
        &= T_1\cdots T_kX_{k+1}(X_1\cdots X_k\epsilon_k(w))\\
        &= T_1\cdots T_k X_1\cdots X_{k+1} \epsilon_k(w) \\
        &=  X_1\cdots X_{k+1} (T_1\cdots T_k \epsilon_k(w)) \\
        &=  X_1\cdots X_{k+1} (T_1\cdots T_k \epsilon_{k+1}\epsilon_k(w)) \\
        &= X_1\cdots X_{k+1}\epsilon_{k+1} (T_1\cdots T_k \epsilon_k(w)) \in L_{k+1}.\\
    \end{align*}

    Lastly, we look at $d_{-}.$ We suppose $v \in L_{k+1}$ say, $v = X_1\cdots X_{k+1}\epsilon_{k+1}(w)$ for $w \in V.$ We get that
    \begin{align*}
        &d_{-}(v) \\
        &= (q-1)(1+qT_{k+1}^{-1}+\ldots + q^{n-k-1}T_{n-1}^{-1}\cdots T_{k+1}^{-1})(v)\\
        &= (q-1)(1+qT_{k+1}^{-1}+\ldots + q^{n-k-1}T_{n-1}^{-1}\cdots T_{k+1}^{-1})(X_1\cdots X_{k+1}\epsilon_{k+1}(w))\\
        &= (q-1)(1+qT_{k+1}^{-1}+\ldots + q^{n-k-1}T_{n-1}^{-1}\cdots T_{k+1}^{-1})(\epsilon_{k+1}(X_1\cdots X_{k+1}w))\\
        &= (q-1)(1+q+\ldots + q^{n-k-1})\epsilon_{k}(X_1\cdots X_{k+1}w) \\
        &= X_1\cdots X_k\epsilon_{k}((q^{n-k}-1)X_{k+1}w) \in L_{k}.\\
    \end{align*}
\end{proof}

Now we will show that the collection of operators $T_i,z_j,d_{-},d_{+}$ acting on the space $L_{\bullet}$ generates an action of $\mathbb{B}_{q,t}.$

\begin{prop}\label{finite rank main construction}
    $L_{\bullet}$ is a $\mathbb{B}_{q,t}^{(n)}$-module.
\end{prop}
\begin{proof}
    We will show that the operators $T_i,z_j,d_{-},d_{+}$ on $L_{\bullet}$ defined in Definition \ref{operators construction} satisfy the relations in Definition \ref{B_qt relations}. Note first that the relations involving only $T_i$'s and $z_j$'s follow immediately from their definition and the fact that $V$ is a $\sD^{+}_n$-module. 

    We will start by verifying the relations between $d_{+}$ and the $T_i.$ We will for the remainder of this proof let $v \in L_k$ and specify various conditions on $k$ as needed. Suppose $0\leq k \leq n-1.$ Then for $1 \leq i \leq k-1$ using the braid relations we see
    \begin{align*}
        & d_{+}T_i(v) \\
        &= q^{k}X_1T_1^{-1}\cdots T_k^{-1}T_i(v) \\
        &= q^{k}X_1T_{i+1}T_1^{-1}\cdots T_k^{-1}(v)\\
        &= T_{i+1}(q^{k}X_1T_1^{-1}\cdots T_k^{-1})(v)\\
        &= T_{i+1}d_{+}(v).\\
    \end{align*}

    Now if $0\leq k \leq n-2$ we see from the braid relations and the fact that $T_{k+1}(v) = v$ 
    \begin{align*}
        &T_1^{-1}d_{+}^{2}(v) \\
        &= T_1^{-1} d_{+}(q^{k}X_1T_1^{-1}\cdots T_k^{-1}(v))\\
        &= T_1^{-1} (q^{k+1}X_1T_1^{-1}\cdots T_{k+1}^{-1})(q^{k}X_1T_1^{-1}\cdots T_k^{-1}(v))\\
        &= q^{2k+1}T_1^{-1}X_1T_{1}^{-1}T_2^{-1}\cdots T_{k+1}^{-1}X_1T_1^{-1}\cdots T_k^{-1}(v)\\
        &= q^{2k}X_2T_2^{-1}\cdots T_{k+1}^{-1} X_1T_1^{-1}\cdots T_k^{-1}(v)\\
        &= q^{2k}X_1X_2T_2^{-1}\cdots T_{k+1}^{-1}T_1^{-1}\cdots T_k^{-1}(v)\\
        &= q^{2k}X_1X_2T_2^{-1}\cdots T_{k+1}^{-1}T_1^{-1}\cdots T_k^{-1}T_{k+1}(v)\\
        &= q^{2k}X_1X_2T_1T_2^{-1}\cdots T_{k+1}^{-1}T_1^{-1}\cdots T_k^{-1}(v)\\
        &= q^{2k+1}X_1T_1^{-1}X_1T_2^{-1}\cdots T_{k+1}^{-1}T_1^{-1}\cdots T_k^{-1}(v) \\
        &= q^{2k+1}X_1T_1^{-1}T_2^{-1}\cdots T_{k+1}^{-1}X_1T_1^{-1}\cdots T_k^{-1}(v)\\
        &= (q^{k+1}X_1T_1^{-1}\cdots T_{k+1}^{-1})(q^{k}X_1T_1^{-1}\cdots T_k^{-1}(v))\\
        &= d_{+}^{2}(v) .\\
    \end{align*}

We will now show that $d_{+}z_i = z_{i+1}d_{+}.$ Suppose $1\leq i \leq k \leq n-1.$ Then we have by using Remark \ref{additional relations remark} 
\begin{align*}
    & z_{i+1}d_{+}(v) \\
    &= (qt)^{-1}q^{k}Y_{i+1}X_1T_{1}^{-1}\cdots T_{k}^{-1}(v)\\
    &= (qt)^{-1}q^{k}Y_{i+1}X_1T_{1}^{-1}\cdots T_{k}^{-1}T_{k+1}^{-1}\cdots T_{n-1}^{-1}(v)\\
    &= (qt)^{-1}q^{k}Y_{i+1}(X_1T_{1}^{-1} \cdots T_{n-1}^{-1})(v)\\
    &= (qt)^{-1}q^{k}(X_1T_{1}^{-1}\cdots T_{n-1}^{-1}) Y_{i}(v)\\
    &= (qt)^{-1}q^{k}X_1T_{1}^{-1}\cdots T_{k}^{-1}Y_{i} T_{k+1}^{-1}\cdots T_{n-1}^{-1}(v) \\
    &= q^{k}X_1T_{1}^{-1}\cdots T_{k}^{-1}(qt)^{-1}Y_{i}(v) \\
    &= d_{+}(z_i(v)).\\
\end{align*}

Next we note that the relations between just $d_{-}$ and the $T_i$ follow trivially from the fact that $d_{-}:L_{k+1}\rightarrow L_{k}$ is a scalar multiple of $\epsilon_{k}|_{L_{k+1}}$ which follows from the relations (see Remark \ref{additional relations remark}). Further, the relation $z_id_{-} = d_{-}z_i$ also follows easily from the fact that $Y_iT_j = T_jY_i$ for $i \notin \{j,j+1\}.$

Now we are left to show that the relations involving $\varphi:= \frac{1}{q-1} [d_{+},d_{-}]$ hold. Notice that $\varphi$ may be computed for $1\leq k \leq n-1$ as 
\begin{align*}
    & (q-1)\varphi(v) \\
    &=  [d_{+},d_{-}](v)\\
    &=  (d_{+}d_{-}-d_{-}d_{+})(v)\\
    &=  d_{+}((q-1)(1+qT_{k}^{-1}+\ldots + q^{n-k}T_{n-1}^{-1}\cdots T_{k}^{-1})(v)) - d_{-}(q^{k}X_1T_1^{-1}\cdots T_k^{-1}v)\\
    &= (q-1)(q^{k-1}X_1T_1^{-1}\cdots T_{k-1}^{-1})(1+qT_{k}^{-1}+\ldots + q^{n-k}T_{n-1}^{-1}\cdots T_{k}^{-1})(v)\\
    &~~~-(q-1)(1+qT_{k+1}^{-1}+\ldots + q^{n-k-1}T_{n-1}^{-1}\cdots T_{k+1}^{-1})(q^{k}X_1T_1^{-1}\cdots T_k^{-1}v)\\
    &= (q-1)q^{k-1}X_1T_1^{-1}\cdots T_{k-1}^{-1}\left((1+qT_{k}^{-1}+\ldots + q^{n-k}T_{n-1}^{-1}\cdots T_{k}^{-1}) - q(1+qT_{k+1}^{-1}+\ldots + q^{n-k-1}T_{n-1}^{-1}\cdots T_{k+1}^{-1}   )T_k^{-1}\right)\\
    &= (q-1)q^{k-1}X_1T_1^{-1}\cdots T_{k-1}^{-1} \\
\end{align*}
    so that $$\varphi(v) = q^{k-1}X_1T_1^{-1}\cdots T_{k-1}^{-1}(v).$$

Let $2\leq k \leq n.$ 
Then 
    \begin{align*}
        &q\varphi d_{-}(v) \\
        &= q\varphi (q-1)(1+qT_{k}^{-1}+ \ldots + q^{n-k}T_{n-1}^{-1}\cdots T_{k}^{-1})v \\
        &= q(q-1)q^{k-2}X_1T_1^{-1}\cdots T_{k-2}^{-1}(1+qT_{k}^{-1}+ \ldots + q^{n-k}T_{n-1}^{-1}\cdots T_{k}^{-1})v \\
        &= (q-1)(1+qT_{k}^{-1}+ \ldots + q^{n-k}T_{n-1}^{-1}\cdots T_{k}^{-1}) q^{k-1}X_1T_1^{-1}\cdots T_{k-2}^{-1}(v)\\
        &= d_{-}\varphi T_{k-1}(v).\\
    \end{align*}

 Let us now show that $T_1 \varphi d_{+} = q d_{+} \varphi.$ Suppose $1 \leq k \leq n-2.$  Then 
 \begin{align*}
     &T_1\varphi d_{+}(v) \\
     &T_1 \varphi (q^{k}X_1T_1^{-1}\cdots T_{k}^{-1})(v) \\
     &= T_1(q^{k}X_1T_1^{-1}\cdots T_{k}^{-1})(q^{k}X_1T_1^{-1}\cdots T_{k}^{-1})(v)\\
     &= q^{2k}T_1X_1T_1^{-1}\cdots T_{k}^{-1}X_1T_1^{-1}\cdots T_k^{-1}(v)\\
     &= q^{2k}T_1(XT_1^{-1}\cdots T_{k}^{-1})^2(v) \\
     &= q^{2k}(X_1T_1^{-1}\cdots T_{k}^{-1})^2T_{k}(v) \\
     &= q^{2k}X_1T_1^{-1}\cdots T_{k}^{-1}X_1T_1^{-1}\cdots T_{k-1}^{-1}(v)\\
     &= q(q^{k}X_1T_{1}^{-1}\cdots T_k^{-1})(q^{k-1}X_1T_1^{-1}\cdots T_{k-1}^{-1})(v)\\
     &= qd_{+}\varphi (v).\\
 \end{align*}

Lastly, we show that $z_1(qd_{+}d_{-} - d_{-}d_{+}) = qt(d_{+}d_{-} - d_{-}d_{+})z_k.$ Take $1 \leq k \leq n-1.$ Then we find
\begin{align*}
    &z_1(qd_{+}d_{-} - d_{-}d_{+})(v) \\
    &= z_1(qd_{+}d_{-}(v) - d_{-}d_{+}(v)) \\
    &= z_1\left( qd_{+}(q-1)(1+qT_k^{-1}+\ldots + q^{n-k}T_{n-1}^{-1}\cdots T_k^{-1})(v) - d_{-}(q^{k}X_1T_1^{-1}\cdots T_k^{-1})(v)\right)\\
    &= (qt)^{-1}Y_1q(1-q)(q^{k-1}X_1T_1^{-1}\cdots T_{k-1}^{-1})(1+qT_k^{-1}+\ldots + q^{n-k}T_{n-1}^{-1}\cdots T_k^{-1})(v)\\
    &~~~ -(qt)^{-1}Y_1(q-1)(1+qT_{k+1}^{-1}+\ldots + q^{n-k-1}T_{n-1}^{-1}\cdots T_{k+1}^{-1})(q^{k}X_1T_1^{-1}\cdots T_{k}^{-1})(v)\\
    &= (qt)^{-1}(q-1)q^{k}Y_1X_1T_1^{-1}\cdots T_{k-1}^{-1}\left( 1+qT_{k+1}^{-1}+\ldots + q^{n-k-1}T_{n-1}^{-1}\cdots T_{k+1}^{-1} - (1+qT_{k+1}^{-1}+\ldots + q^{n-k-1}T_{n-1}^{-1}\cdots T_{k+1}^{-1}T_k^{-1})\right)(v)\\
    &= (qt)^{-1}(q-1)q^{k}Y_1X_1T_1^{-1}\cdots T_{k-1}^{-1}\left(1 + (q-1)(T_{k}^{-1}+qT_{k+1}^{-1}T_k^{-1}+\ldots + q^{n-k-1}T_{n-1}^{-1}\cdots T_k^{-1}) \right)(v)\\
    &= (qt)^{-1}(q-1)q^{k}Y_1X_1T_1^{-1}\cdots T_{k-1}^{-1}\left(q^{n-k}T_k^{-1}\cdots T_{n-1}^{-1}T_{n-1}^{-1}\cdots T_{k}^{-1} \right)(v)\\
    &= (qt)^{-1}q^{n}(q-1)Y_1\widetilde{\pi} T_{n-1}^{-1}\cdots T_k^{-1}(v)\\
    &= (qt)^{-1}q^{n}(q-1)\widetilde{\pi} (tY_n)T_{n-1}^{-1}\cdots T_k^{-1}(v)\\
    &= (qt)^{-1}tq^{n}(q-1)X_1T_{1}^{-1}\cdots T_{k-1}^{-1}(T_k^{-1}\cdots T_{n-1}^{-1}Y_nT_{n-1}^{-1}\cdots T_k^{-1})(v)\\
    &= (qt)^{-1}tq^{n}(q-1)X_1T_{1}^{-1}\cdots T_{k-1}^{-1}(q^{-(n-k)}Y_k)(v)\\
    &= (qt)^{-1}tq^{k}(q-1)X_1T_1^{-1}\cdots T_{k-1}^{-1}Y_k(v)\\
    &= (qt)^{-1}qt(q-1) \left(q^{k-1}X_1T_1^{-1}\cdots T_{k-1}^{-1} \right)Y_k(v)\\
    &= qt(q-1)\varphi z_k(v) \\
    &= qt [d_{+},d_{-}] z_k(v).\\
\end{align*}
    
\end{proof}

\begin{cor}\label{finite rank functorial}
    The map $W \rightarrow L_{\bullet}(W)$ is a covariant functor $\sD_{n}^{+}-\MOD \rightarrow \mathbb{B}_{q,t}^{(n)}-\MOD.$
\end{cor}

\begin{proof}
    Suppose $\phi: U \rightarrow W$ is a homogeneous $\sD_{n}^{+}$-module map. Now for any $0 \leq k \leq n$ we see that if $v = X_1\cdots X_k \epsilon_k(u)\in L_k(U)$ then 
    $$\phi(v) = \phi(X_1\cdots X_k \epsilon_k(u)) = X_1\cdots X_k \epsilon_k(\phi(u)) \in L_k(W).$$
    Thus $\phi$ yields a map $\phi_{\bullet}: L_{\bullet}(U) \rightarrow L_{\bullet}(W)$ given by restricting $\phi$ to each of the subspaces $L_k(U) \subset U.$ From Definition \ref{operators construction} we see that each of the operators $T_i,z_i,d_{-},d_{+}$ is expressed \textit{entirely} in terms of the action of $\sD_n^{+}$ on $U$ and as such we conclude that $\phi_{\bullet}$ is a  $\mathbb{B}_{q,t}^{(n)}$ module map.
\end{proof}

\subsection{The Polynomial Case}
The goal of this section is to relate the $\mathbb{B}_{q,t}^{(n)}$ modules $L_{\bullet}(W)$ constructed above to the polynomial representation $V^{\pol}_{\bullet}$ of $\mathbb{B}_{q,t}$ in the case when $W = V_{\pol}^{(n)}.$ We will show that there are natural maps $x_1\cdots x_k\mathbb{Q}(q,t)[x_1,\ldots,x_k]\otimes \Lambda \rightarrow x_1\cdots x_k\mathbb{Q}(q,t)[x_1,\ldots,x_n]^{\mathfrak{S}_{(1^k,n-k)}}$ which are $\mathbb{B}_{q,t}^{(n)}$ module projections. This is nontrivial since the definitions of $z_i$ and $d_{-}$ are quite different in both modules. We will use the work of Ion-Wu to bridge this gap.

\begin{defn}\cite{Ion_2022}
    Let $\mathscr{P}_{as}^{+}:= \mathbb{Q}(q,t)[x_1,x_2,\ldots]\otimes \Lambda$ denote the space of \textbf{\textit{almost symmetric functions }}with distinguished subspaces $\mathbb{Q}(q,t)[x_1,\ldots,x_k]\otimes \Lambda$ of functions which are symmetric after k. We will write $\mathfrak{X}_{k} := x_{k+1}+x_{k+2}+\ldots.$ Let $\Xi^{(n)}:\mathscr{P}_{as}^{+} \rightarrow \mathbb{Q}(q,t)[x_1,\ldots, x_n]$ be the projection homomorphism defined by $\Xi^{(n)}(x_{n+i})= 0$ for all $i \geq 0.$ By abuse of notation we will also denote by $\Xi^{(n)}$ any projection map defined on polynomials defined by restricting $\Xi^{(n)}$ to any subspace of $\mathscr{P}_{as}^{+}$ notably including $\mathbb{Q}(q,t)[x_1,\cdots,x_{n},x_{n+1}]$ and $x_1\cdots x_k \mathbb{Q}(q,t)[x_1,\ldots,x_k]\otimes \Lambda.$ Let $\mathscr{P}_{\bullet}:= \bigoplus_{k\geq 0} \mathbb{Q}(q,t)[x_1,\ldots,x_k]\otimes \Lambda.$ Consider the following operators given on $f \in \mathbb{Q}(q,t)[x_1,\ldots,x_k]\otimes \Lambda$ as follows:

    \begin{itemize}
        \item $T_i(f) := s_i(f) + (1-q)x_i \frac{f-s_i(f)}{x_i-x_{i+1}}$ 
        \item $d_{+}(f) := q^{k}X_1T_{1}^{-1}\cdots T_{k}^{-1}(f)$
        \item $d_{-}(x_1^{a_1}\cdots x_{k+1}^{a_{k+1}}F[\mathfrak{X}_{k+1}]):= -x_1^{a_1}\cdots x_{k}^{a_{k}}u^{a_{k+1}} F[\mathfrak{X}_k-u]\Exp[(1-q)u^{-1}\mathfrak{X}_k]~|_{u^0}$
        \item $z_i(f) = (qt)^{-1}\mathscr{Y}_i(f) :=
 \lim_{m} (qt)^{-1} \widetilde{Y}_i^{(m)} \Xi^{(m)}(f)$ 
    \end{itemize}
    where $\mathscr{Y}_i$ are the \textit{limit} Cherednik operators, $\lim_m$ is the limit as defined by Ion-Wu (with $q$ and $t$ swapped) and $\widetilde{Y}_i^{(m)}$ are the \textit{deformed} Cherednik operators.
\end{defn}

We can use the work of Ion-Wu to relate the above $\mathbb{B}_{q,t}$ module $\mathscr{P}_{\bullet}$ to the $\mathbb{B}_{q,t}$ module $W^{\pol}_{\bullet}$ defined by Carlsson-Gorsky-Mellit as follows.

\begin{thm*}\label{Ion-Wu to CGM}\cite{Ion_2022}
    The maps $T_i,d_{+},d_{-},z_i$ on $\mathscr{P}_{\bullet}$ define a representation of $\mathbb{B}_{q,t}.$ This representation is isomorphic to the $\mathbb{B}_{q,t}$ representation on $W^{\pol}_{\bullet}:= \bigoplus_{k\geq 0} (y_1\cdots y_{k})^{-1}V^{\pol}_{k}$ defined by Carlsson-Gorsky-Mellit via the map $\Phi_{\bullet} = \bigoplus_{k \geq 0}\Phi_k$ defined by 
    $$\Phi_{k}(x_1^{a_1}\cdots x_k^{a_k}F[\mathfrak{X}_k]):= y_1^{a_1-1}\cdots y_k^{a_k-1}F\left[\frac{X}{q-1}\right].$$
\end{thm*}

\begin{remark}
    Ion-Wu in their paper also construct the additional operator $d_{+}^{*}$ on $\mathscr{P}_{\bullet}$ from which they obtain an action of $\mathbb{A}_{q,t}$ on $\mathscr{P}_{\bullet}$. Further, they show that this $\mathbb{A}_{q,t}$ module is isomorphic to the standard $\mathbb{A}_{q,t}$ representation as defined by Mellit \cite{M2021} which is the same as the Carlsson-Gorsky-Mellit action of $\mathbb{A}_{q,t}$ on $W_{\bullet}^{\pol}.$ The result as stated above is thus a strictly weaker result than the main theorem of Ion-Wu but as we are only interested in the subalgebra $\mathbb{B}_{q,t}$ of $\mathbb{A}_{q,t}$, we will only require the above result as stated.
\end{remark}

By the above theorem of Ion-Wu we find that each of the spaces $y_1\cdots y_k W_k^{\pol} = V_k^{\pol}$ gets mapped by $\Phi_k^{-1}$ to the space $x_1\cdots x_k \mathbb{Q}(q,t)[x_1,\ldots,x_k] \otimes \Lambda$ which we will call $L_k^{\pol}.$ Thus we see that $\mathbb{B}_{q,t}$ acts on the space $L_{\bullet}^{\pol}:= \bigoplus_{k\geq 0} L_k^{\pol} \subset \sP_{\bullet}.$ For all $n \geq 0$ can relate $L_{\bullet}^{\pol}$ to $L_{\bullet}(V_{\pol}^{(n)})$ in the following way:

\begin{prop}\label{quotient map prop for poly}
    The map $\Xi_{\bullet}^{(n)}:L_{\bullet}^{\pol} \rightarrow L_{\bullet}(V_{\pol}^{(n)})$ defined component-wise by $\Xi^{(n)}$ is a $\mathbb{B}_{q,t}^{(n)}$ module map.
\end{prop}

\begin{proof}
    We need to show that $\Xi^{(n)}$ commutes with the operators $T_i,z_i,d_{+},d_{-}$ as defined on both of the spaces $L_{\bullet}^{\pol}$ and $L_{\bullet}(V_{\pol}^{(n)})$ respectively. For $T_i$ and $d_{+}$ this is immediate. For $z_i$ we note that from the construction of the deformed Cherednik operators $\widetilde{Y}_{i}^{(n)}$ \cite{Ion_2022} we have that if $1 \leq i \leq k$ then 
    $$\widetilde{Y}_{i}^{(n)}X_i = Y_i^{(n)} X_i.$$ Further, from Ion-Wu we also know that (using this paper's conventions) for all $n \geq i$
    $$Y_i^{(n)} X_i ~\Xi^{(n)} = \Xi^{(n)}  Y_i^{(n+1)} X_i.$$ Thus for $f \in L_{k}^{\pol}$ and $1\leq i \leq k$ we find
    $$Y_i^{(n)}\Xi^{(n)}(f) = \Xi^{(n)} \mathscr{Y}_i(f) $$ and so 
    $$z_i\Xi^{(n)} = \Xi^{(n)} z_i.$$

Let $0\leq k \leq n-1$ and $f = x_1\cdots x_{k+1} g$ for $g \in \mathbb{Q}(q,t)[x_1,\ldots, x_k].$ From the author's prior paper \cite{MBWArxiv}, we know that 

\begin{align*}
    d_{-}(f) &= -\lim_{m}\left( \frac{1+qT_{k+1}^{-1}+\ldots + q^{m-k-1}T_{m-1}^{-1}\cdots T_{k+1}^{-1}}{1+q+\ldots + q^{m-k-1}} \right) \Xi^{(m)}(f) \\
    &= -\left( \frac{1}{1+q+q^2+\ldots} \right) \lim_{m}\left(1+qT_{k+1}^{-1}+\ldots + q^{m-k-1}T_{m-1}^{-1}\cdots T_{k+1}^{-1}\right)\Xi^{(m)}(f) \\
    &= -(1-q) \lim_{m}\left(1+qT_{k+1}^{-1}+\ldots + q^{m-k-1}T_{m-1}^{-1}\cdots T_{k+1}^{-1}\right)\Xi^{(m)}(f) \\
    &= (q-1) \lim_{m} \left(1+qT_{k+1}^{-1}+\ldots + q^{m-k-1}T_{m-1}^{-1}\cdots T_{k+1}^{-1}\right)X_1\cdots X_{k+1}\Xi^{(m)}(g).
\end{align*}

Now if $m \geq k$ then 
\begin{align*}
    &\Xi^{(m)}\left(1+qT_{k+1}^{-1}+\ldots + q^{m-k}T_{m}^{-1}\cdots T_{k+1}^{-1}\right)X_1\cdots X_{k+1}\\
    &\Xi^{(m)}\left(1+qT_{k+1}^{-1}+\ldots + q^{m-k-1}T_{m-1}^{-1}\cdots T_{k+1}^{-1}\right)X_1\cdots X_{k+1} + \Xi^{(m)}q^{m-k}T_{m}^{-1}\cdots T_{k+1}^{-1}X_1\cdots X_{k+1} \\
    &= \left(1+qT_{k+1}^{-1}+\ldots + q^{m-k-1}T_{m-1}^{-1}\cdots T_{k+1}^{-1}\right)X_1\cdots X_{k+1} \Xi^{(m)} + \Xi^{(m)}X_{m+1}T_{m}\cdots T_{k+1}X_1\cdots X_{k} \\
    &= \left(1+qT_{k+1}^{-1}+\ldots + q^{m-k-1}T_{m-1}^{-1}\cdots T_{k+1}^{-1}\right)X_1\cdots X_{k+1} \Xi^{(m)}. \\
\end{align*}

Therefore, 
\begin{align*}
    &\Xi^{(n)}(d_{-}(f)) \\
    &= (q-1) \left(1+qT_{k+1}^{-1}+\ldots + q^{n-k-1}T_{n-1}^{-1}\cdots T_{k+1}^{-1}\right)X_1\cdots X_{k+1}\Xi^{(n)}(g)\\
    &= (q-1) \left(1+qT_{k+1}^{-1}+\ldots + q^{n-k-1}T_{n-1}^{-1}\cdots T_{k+1}^{-1}\right)\Xi^{(n)}(X_1\cdots X_{k+1}g) \\
    &= (q-1) \left(1+qT_{k+1}^{-1}+\ldots + q^{n-k-1}T_{n-1}^{-1}\cdots T_{k+1}^{-1}\right)\Xi^{(n)}(f)\\
    &= d_{-}\Xi^{(n)}(f).\\
\end{align*}

Thus $\Xi^{(n)}d_{-} = d_{-}\Xi^{(n)}$ and so $\Xi_{\bullet}^{(n)}$ is a $\mathbb{B}_{q,t}^{(n)}$ module map.
\end{proof}

\begin{remark}
    Since $L_{\bullet}^{\pol}$ is isomorphic as a $\mathbb{B}_{q,t}$ module to $V_{\bullet}^{\pol}$ via the map $\Phi_{\bullet}$ and from Proposition \ref{quotient map prop for poly} we know that $\Xi^{(n)}:L_{\bullet}^{\pol} \rightarrow L_{\bullet}(V_{\pol}^{(n)})$ is a $\mathbb{B}_{q,t}^{(n)}$ module quotient, it follows that $L_{\bullet}(V_{\pol}^{(n)})$ is a $\mathbb{B}_{q,t}^{(n)}$ module quotient of $\Res^{\mathbb{B}_{q,t}}_{\mathbb{B}_{q,t}^{(n)}} V_{\bullet}^{\pol}.$
\end{remark}

\subsection{$\mathbb{B}_{q,t}$ Modules From Compatible Sequences}

We will now build representations for the full $\mathbb{B}_{q,t}$ algebra given certain special families of DAHA representations. 

\begin{defn}
    Let $C= \left( (V^{(n)})_{n \geq n_1}, (\Pi^{(n)})_{n\geq n_1}\right)$ be a collection of $\mathbb{Q}(q,t)$-vector spaces and maps $\Pi^{(n)}: V^{(n+1)}\rightarrow V^{(n)}$ with $n_1 \geq 1.$ We call $C$ a \textit{\textbf{compatible sequence}} if the following conditions hold:
    \begin{itemize}
        \item Each $V^{(n)}$ is a graded $\sD_n^{+}$-module
        \item The maps $\Pi^{(n)}: V^{(n+1)}\rightarrow V^{(n)}$ are degree-preserving.
        \item Each map $\Pi^{(n)}$ is a $\sA_{n}^{X}$ module map.
        \item $\Pi^{(n)}X_{n+1} = 0$
        \item $\Pi^{(n)}\pi^{(n+1)}T_n = \pi^{(n)}\Pi^{(n)}.$
    \end{itemize}

Given compatible sequences $C= \left( (V^{(n)})_{n \geq n_1}, (\Pi^{(n)})_{n\geq n_1}\right)$ and $D= \left( (W^{(n)})_{n \geq n_2}, (\Psi^{(n)})_{n\geq n_2}\right)$ a homomorphism $\phi: C \rightarrow D$ is a collection of maps $\phi = (\phi^{(n)})_{n \geq \max(n_1,n_2)}$ with $\phi^{(n)}: V^{(n)} \rightarrow W^{(n)}$ such that 
\begin{itemize}
    \item $\phi^{(n)}$ are degree-preserving $\sD_n^{+}$ module maps.
    \item $\phi^{(n)}\Pi^{(n)} = \Psi^{(n)}\phi^{(n+1)}.$
\end{itemize}

We will write $\mathfrak{C}$ for the category of compatible sequences.
\end{defn} 

\begin{remark}
    The importance of the relation $\Pi^{(n)}\pi^{(n+1)}T_n = \pi^{(n)}\Pi^{(n)}$ can be traced back to at least the work of Ion-Wu \cite{Ion_2022} on their stable-limit DAHA. This relation allowed Ion-Wu to construct the limit Cherednik operators on the space of almost symmetric functions utilizing a remarkable stability relation for the classical Cherednik operators. We will be following a similar idea in a different setting in this section of the paper.
    
    This relation may be interpreted as relating to the natural inclusion map on extended affine symmetric groups $\widehat{\mathfrak{S}}_n \rightarrow \widehat{\mathfrak{S}}_{n+1}$ given by $s_i \rightarrow s_i$ for $1\leq i \leq n-1$ and $\pi \rightarrow \pi s_{n}.$ Diagrammatically, this map sends the crossing diagram for some $\sigma \in \widehat{\mathfrak{S}}_n$ on $n$-strands to the corresponding crossing diagram on $(n+1)$-strands where we send $n+1$ to itself. 
\end{remark}

For the remainder of this section we fix a compatible sequence $C= \left( (V^{(n)})_{n \geq n_0}, (\Pi^{(n)})_{n\geq n_0}\right).$ It is easy to check that for $0 \leq k \leq n,$ $\Pi^{(n)}(L_{k}(V^{(n+1)})) \subset L_{k}(V^{(n)})$ so that the following definition makes sense.

\begin{defn}
    For $k \geq 0$ define $\fL_k = \fL_k(C)$ to be the stable-limit
    $\fL_k:= \lim_{\leftarrow} L_k(V^{(n)})$
    with respect to the maps $\Pi^{(n)}.$
    We define $\fL_{\bullet} = \fL_{\bullet}(C)$ as
    $\fL_{\bullet} = \bigoplus_{k \geq 0} \fL_{k}.$
    We will write $\Pi^{(n)}_{\bullet}: L_{\bullet}(V^{(n+1)}) \rightarrow L_{\bullet}(V^{(n)}) $ for the map obtained by restricting $\Pi^{(n)}$ to each component $L_{k}(V^{(n+1)}).$
\end{defn}

If we let $\widetilde{V}$ denote the stable-limit of the spaces $V^{(n)}$ with respect to the maps $\Pi^{(n)}$ then we can reinterpret the spaces $\fL_k$ as 
$$\fL_k = \{ v \in X_1\cdots X_k \widetilde{V} |~ T_i(v) = v~~ \text{for}~~ i > k \}.$$ 

\begin{lem}\label{stability lemma}
    For $n \geq n_0$ the map $\Pi^{(n)}_{\bullet}: L_{\bullet}(V^{(n+1)}) \rightarrow L_{\bullet}(V^{(n)})$ is a $\mathbb{B}_{q,t}^{(n)}$-module map.
\end{lem}

\begin{proof}
    By definition $\Pi^{(n)}$ is a $\sA_n^{X}$-module map so for $1 \leq i \leq k-1,$
    $\Pi^{(n)}T_i = T_i \Pi^{(n)}.$ Further, we also know that if $k \leq n-1$ then $\Pi^{(n)}d_{+} = d_{+} \Pi^{(n)}$ since on $L_k,$ $d_{+} = q^{k}X_1T_1^{-1}\cdots T_k^{-1}.$ 

    Now let $1 \leq k \leq n$ and $v \in L_k$ say, $v = X_1\cdots X_k \epsilon_k(w).$
    We see that from a nearly identical calculation to one seen in the proof of Proposition \ref{quotient map prop for poly}
    \begin{align*}
        & \Pi^{(n)} d_{-}(v)\\
        &= \Pi^{(n)}(q-1)(1+qT_{k}^{-1}+\ldots + q^{n+1-k}T_{n}^{-1}\cdots T_{k}^{-1})(v)\\
        &= \Pi^{(n)}(q-1)(1+qT_{k}^{-1}+\ldots + q^{n+1-k}T_{n}^{-1}\cdots T_{k}^{-1})(X_1\cdots X_k \epsilon_k(w))\\
        &= \Pi^{(n)}(q-1)(1+qT_{k}^{-1}+\ldots + q^{n-k}T_{n-1}^{-1}\cdots T_{k}^{-1})(X_1\cdots X_k \epsilon_k(w)) + \Pi^{(n)}(q-1)q^{n+1-k}T_n^{-1}\cdots T_k^{-1}X_1\cdots X_k \epsilon_k(w)\\
        &= (q-1)(1+qT_{k}^{-1}+\ldots + q^{n-k}T_{n-1}^{-1}\cdots T_{k}^{-1})\Pi^{(n)}(X_1\cdots X_k \epsilon_k(w)) + \Pi^{(n)}(q-1)X_{n+1}T_n\cdots T_kX_1\cdots X_{k-1} \epsilon_k(w)\\
        &= (q-1)(1+qT_{k}^{-1}+\ldots + q^{n-k}T_{n-1}^{-1}\cdots T_{k}^{-1})\Pi^{(n)}(X_1\cdots X_k \epsilon_k(w))\\
        &= (q-1)(1+qT_{k}^{-1}+\ldots + q^{n-k}T_{n-1}^{-1}\cdots T_{k}^{-1})\Pi^{(n)}(v)\\
        &= d_{-}\Pi^{(n)}(v).\\
    \end{align*}

    Lastly, let $1 \leq i \leq k.$ Using the relation $\Pi^{(n)}\pi^{(n+1)}T_n = \pi^{(n)}\Pi^{(n)}$ we find
    \begin{align*}
        & \Pi^{(n)}z_i(v) \\
        &= \Pi^{(n)}(qt)^{-1}Y_i(v) \\
        &= (qt)^{-1}\Pi^{(n)}q^{n-i+2}T_{i-1}\cdots T_1\pi^{(n+1)}T_{n}^{-1}\cdots T_i^{-1}(v)\\
        &= (qt)^{-1}\Pi^{(n)}q^{n-i+2}T_{i-1}\cdots T_1\pi^{(n+1)}T_{n}^{-1}\cdots T_i^{-1}(X_1\cdots X_k \epsilon_k(w))\\
        &= (qt)^{-1}\Pi^{(n)}q^{n-i+2}T_{i-1}\cdots T_1\pi^{(n+1)}T_{n}^{-1}\cdots T_i^{-1}X_i(X_1\cdots X_{i-1}X_{i+1}\cdots X_k \epsilon_k(w))\\
        &= (qt)^{-1}\Pi^{(n)}q^{2}T_{i-1}\cdots T_1\pi^{(n+1)}X_{n+1}T_{n}\cdots T_i(X_1\cdots X_{i-1}X_{i+1}\cdots X_k \epsilon_k(w))\\
        &= (qt)^{-1}\Pi^{(n)}q^{2}tT_{i-1}\cdots T_1X_1\pi^{(n+1)}T_{n}\cdots T_i(X_1\cdots X_{i-1}X_{i+1}\cdots X_k \epsilon_k(w))\\
        &= (qt)^{-1}q^{2}tT_{i-1}\cdots T_1X_1\Pi^{(n)}\pi^{(n+1)}T_{n}\cdots T_i(X_1\cdots X_{i-1}X_{i+1}\cdots X_k \epsilon_k(w))\\
        &= (qt)^{-1}q^{2}tT_{i-1}\cdots T_1X_1\pi^{(n)}\Pi^{(n)}T_{n-1}\cdots T_i(X_1\cdots X_{i-1}X_{i+1}\cdots X_k \epsilon_k(w))\\
        &= (qt)^{-1}q^{2}T_{i-1}\cdots T_1\pi^{(n)}X_nT_{n-1}\cdots T_i(X_1\cdots X_{i-1}X_{i+1}\cdots X_k\Pi^{(n)} (\epsilon_k(w))\\
        &= (qt)^{-1}q^{n-i+1}T_{i-1}\cdots T_1\pi^{(n)}T_{n-1}^{-1}\cdots T_i^{-1} X_{i} (X_1\cdots X_{i-1}X_{i+1}\cdots X_k\Pi^{(n)} (\epsilon_k(w)) \\
        &= (qt)^{-1}q^{n-i+1}T_{i-1}\cdots T_1\pi^{(n)}T_{n-1}^{-1}\cdots T_i^{-1} \Pi^{(n)}(X_1\cdots X_k\epsilon_k(w)) \\
        &= z_i\Pi^{(n)}(v).\\
    \end{align*}
\end{proof}

As an immediate consequence of Lemma \ref{stability lemma} and Lemma \ref{stable-limit of modules} we may make the following definition.

\begin{defn}\label{main stable-limit construction}
    We define the graded $\mathbb{B}_{q,t}$ module structure on $\fL_{\bullet}$ given by the stable-limit of the graded $\mathbb{B}_{q,t}^{(n)}$ modules $L_{\bullet}^{(n)}$ with respect to the maps $\Pi^{(n)}_{\bullet}:L_{\bullet}^{(n+1)}\rightarrow L_{\bullet}^{(n)}.$
\end{defn}

\begin{example}
    In the case of the polynomial representations of $\sD_n^{+}$, $V_{\pol}^{(n)}$, we see using Proposition \ref{quotient map prop for poly} that $\fL_{\bullet}(C_{\pol}) \cong V^{\pol}_{\bullet}$  where 
    $$C_{\pol}:= \left( (V_{\pol}^{(n)})_{n \geq 1}, (\Xi_{\bullet}^{(n)})_{n\geq 1}\right).$$
\end{example}

The construction in Definition \ref{main stable-limit construction} associates to any compatible sequence $C$ a graded module $\fL_{\bullet}(C)$ of $\mathbb{B}_{q,t}.$ We can easily see that this construction is functorial.

\begin{thm}(Main Theorem)\label{main theorem}
    The map $C \rightarrow \fL_{\bullet}(C)$ is a covariant functor $\mathfrak{C} \rightarrow \mathbb{B}_{q,t}-\MOD.$
\end{thm}
\begin{proof}
    This follows immediately using the functoriality described in Remark \ref{functoriality of stable-limit of modules} and from the fact that the operators on $\fL_{\bullet}(C)$ are described \textit{entirely} in terms of the action of each $\sD_n^{+}$ on $V^{(n)}.$
\end{proof}

\begin{remark}
Recently, Gonz\'{a}lez-Gorsky-Simental \cite{gonzález2023calibrated} introduced the extended algebra $\mathbb{B}_{q,t}^{\ext}$ and the notion of \textit{calibrated} $\mathbb{B}_{q,t}^{\ext}$ modules. The extended algebra $\mathbb{B}_{q,t}^{\ext}$ contains additional $\Delta$-operators with specific relations motivated by the $\Delta$-operators in Macdonald theory. Calibrated $\mathbb{B}_{q,t}^{\ext}$ modules are those modules with a basis of joint eigenvectors for the $z_i$'s and the additional operators $\Delta_{p_m}$ with simple nonzero spectrum.

In the case of the polynomial representations of DAHAs, the $\mathbb{B}_{q,t}$ representation $\fL_{\bullet}(C_{\pol})$ has an extended action by $\mathbb{B}_{q,t}^{\ext}$ using $\Delta$-operators and this representation is calibrated. It is an interesting question to figure out exactly which properties of the family of DAHA modules $C_{\pol}$ allow for this extended calibrated action by $\mathbb{B}_{q,t}^{\ext}$. 
\end{remark}

\section{Compatible Sequences From AHA}
 In this section we give a method for defining compatible sequences. We will consider families of representations for the affine Hecke algebras $\sA_n$ in type $GL$ with special properties which we call \textit{pre-compatible}. These families of representations for $\sA_n$ can then be induced to give representations of the corresponding $\sD_n^{+}$ which can be shown to be compatible after carefully defining the correct connecting maps.

\begin{defn}
    Let $C= \left( (U^{(n)})_{n \geq n_0}, (\kappa^{(n)})_{n\geq n_0}\right)$ be a collection of $\mathbb{Q}(q,t)$-vector spaces and maps $\kappa^{(n)}: U^{(n+1)}\rightarrow U^{(n)}$ with $n_1 \geq 1.$ We call $C$ a \textbf{\textit{pre-compatible}} sequence if the following hold:
    \begin{itemize}
        \item Each $U^{(n)}$ is a graded $\sA_n$ module (grading is arbitrary)
        \item The maps $\kappa^{(n)}:U^{(n+1)} \rightarrow U^{(n)}$ are degree preserving $\sH_n$ module maps 
        \item $\kappa^{(n)}\pi^{(n+1)}T_n = \pi^{(n)}\kappa^{(n)}.$
    \end{itemize}

    Given any pre-compatible sequence $C$ we define the spaces $(V_{C}^{(n)})_{n\geq n_0}$ by 
    $$V_{C}^{(n)}:= \Ind_{\sA_n}^{\sD_{n}^{+}} U^{(n)}$$
    which we endow with the grading inherited by $\mathbb{Q}(q,t)[X_1,\ldots, X_n] \otimes U^{(n)}$ (which is isomorphic as a vector space). Define the maps $(\Pi_{C}^{(n)}:V_{C}^{(n+1)} \rightarrow V_{C}^{(n)} )_{n\geq n_0}$ by 
    $$\Pi_{C}^{(n)}(X_1^{\alpha_1}\cdots X_{n+1}^{\alpha_{n+1}} \otimes v):= \mathbbm{1}(\alpha_{n+1} = 0) \otimes \kappa^{(n)}(v).$$ We will write $\Ind(C)$ for the family 
$$\Ind(C):= \left( (V_{C}^{(n)})_{n \geq n_0}, (\Pi_{C}^{(n)})_{n\geq n_0}\right).$$
\end{defn}

\begin{prop}\label{pre-compatible implies compatible}
    If $C$ is pre-compatible then $\Ind(C)$ is compatible.
\end{prop}
\begin{proof}
    By construction each space $V_{C}^{(n)}$ is a graded $\sD_n^{+}$ module and the maps $\Pi_{C}^{(n)}$ are degree preserving $\sA_n^{X}$ maps with $\Pi_{C}^{(n)}X_{n+1} = 0.$ Thus we only need to show that $\Pi_{C}^{(n)}\pi^{n+1}T_n = \pi^{(n)}\Pi_{C}^{(n)}.$ To see this we have the following:

    \begin{align*}
        &\Pi_{C}^{(n)}\pi^{(n+1)}T_n(X_1^{\alpha_1}\cdots X_{n+1}^{\alpha_{n+1}} \otimes v) \\
        &= X_2^{\alpha_1}\cdots X_n^{\alpha_{n-1}}\Pi_{C}^{(n)}\pi^{(n+1)}T_n(X_n^{\alpha_n}X_{n+1}^{\alpha_{n+1}} \otimes v)\\
        &= X_2^{\alpha_1}\cdots X_n^{\alpha_{n-1}}\Pi_{C}^{(n)}\pi^{(n+1)} \left(\left(X_{n}^{\alpha_{n+1}}X_{n+1}^{\alpha_n}T_n + (1-q)X_{n}\frac{X_n^{\alpha_n}X_{n+1}^{\alpha_{n+1}}- X_{n}^{\alpha_{n+1}}X_{n+1}^{\alpha_n}}{X_n - X_{n+1}}\right) \otimes v \right) \\
        &= X_2^{\alpha_1}\cdots X_n^{\alpha_{n-1}}\Pi_{C}^{(n)}\pi^{(n+1)} \left(X_{n}^{\alpha_{n+1}}X_{n+1}^{\alpha_n}\otimes T_n v\right) + (1-q)X_2^{\alpha_1}\cdots X_n^{\alpha_{n-1}}\Pi_{C}^{(n)}\pi^{(n+1)}X_{n}\frac{X_n^{\alpha_n}X_{n+1}^{\alpha_{n+1}}- X_{n}^{\alpha_{n+1}}X_{n+1}^{\alpha_n}}{X_n - X_{n+1}}(1  \otimes v)  \\
        &= X_2^{\alpha_1}\cdots X_n^{\alpha_{n-1}}\Pi_{C}^{(n)}X_{n+1}^{\alpha_{n+1}}(tX_{1})^{\alpha_n} \pi^{(n+1)} \left( 1 \otimes T_n v\right)  \\
        &~~~+ (1-q)X_2^{\alpha_1}\cdots X_n^{\alpha_{n-1}}\Pi_{C}^{(n)}X_{n+1}\pi^{(n+1)}\frac{X_n^{\alpha_n}X_{n+1}^{\alpha_{n+1}}- X_{n}^{\alpha_{n+1}}X_{n+1}^{\alpha_n}}{X_n - X_{n+1}}(1  \otimes v)\\
        &=  X_2^{\alpha_1}\cdots X_n^{\alpha_{n-1}}\Pi_{C}^{(n)}X_{n+1}^{\alpha_{n+1}}(tX_{1})^{\alpha_n} \pi^{(n+1)} \left( 1 \otimes T_n v\right) \\
        &= \mathbbm{1}(\alpha_{n+1} = 0) (tX_{1})^{\alpha_n} X_2^{\alpha_1}\cdots X_n^{\alpha_{n-1}} \Pi_{C}^{(n)}( 1\otimes \pi^{(n+1)}T_n(v)) \\
        &= \mathbbm{1}(\alpha_{n+1} = 0) (tX_{1})^{\alpha_n} X_2^{\alpha_1}\cdots X_n^{\alpha_{n-1}}\otimes \kappa^{(n)}(\pi^{(n+1)}T_n(v))\\
        &= \mathbbm{1}(\alpha_{n+1} = 0) (tX_{1})^{\alpha_n} X_2^{\alpha_1}\cdots X_n^{\alpha_{n-1}}\otimes \pi^{(n)}\kappa^{(n)}(v)\\
        &= \mathbbm{1}(\alpha_{n+1} = 0) (tX_{1})^{\alpha_n} X_2^{\alpha_1}\cdots X_n^{\alpha_{n-1}}\pi^{(n)} \otimes \kappa^{(n)}(v)\\
        &= \pi^{(n)} \left(\mathbbm{1}(\alpha_{n+1} = 0)X_1^{\alpha_1}\cdots X_{n}^{\alpha_n} \otimes \kappa^{(n)}(v) \right)\\
        &= \pi^{(n)}\Pi_{C}^{(n)}(X_1^{\alpha_1}\cdots X_{n+1}^{\alpha_{n+1}} \otimes v).\\
    \end{align*}
    Thus $\Pi_{C}^{(n)}\pi^{(n+1)}T_n = \pi^{(n)}\Pi_{C}^{(n)}$ and so $\Ind(C)$ is compatible. 
\end{proof}

We will now give a large family of pre-compatible sequences built from Young diagrams. The modules in these sequences are the same (up to changing conventions) as the modules in \cite{DL_2011} and the author's prior paper \cite{weising2023murnaghantype}

\begin{defn}
    Define the $\mathbb{Q}(q,t)$-algebra homomorphism $\rho_n:\sA_n \rightarrow \sH_n$ by 
    \begin{itemize}
        \item $\rho_n(T_i) = T_i$ for $1\leq i \leq n-1$
        \item $\rho_n(\pi^{(n)}) = T_1^{-1}\cdots T_{n-1}^{-1}.$ 
    \end{itemize}
    For a $\sH_n$-module $V$ we will denote by $\rho_n^{*}(V)$ the $\sA_n$-module with action defined for $v \in V$ and $X \in \sA_n$ by $X(v) := \rho_n(X)(v).$
\end{defn}

\begin{defn}\label{vv polynomial spaces}
    Let $\lambda$ be a Young diagram. We will write $\SYT(\lambda)$ for the set of Young tableaux of shape $\lambda.$ Let $n_{\lambda}:= |\lambda| + \lambda_1$ and for $n \geq n_{\lambda}$ define $\lambda^{(n)}:= (n-|\lambda|, \lambda).$ Given a box, $\square$, in a Young diagram $\lambda$ we define the \textbf{content} of $\square$ as $c(\square) := a-b$ where $\square = (a,b)$ as drawn in the $\mathbb{N}\times \mathbb{N}$ grid (in \textit{English notation}). We will write $c_{\tau}(i) = c(\square)$ where $\tau(\square) = i.$ We order the set $\SYT(\lambda)$ by defining the cover relation $s_i(\tau) > \tau$ if $c_{\tau}(i)-c_{\tau}(i+1) > 1.$ Define the $\sH_n$-module $S_{\lambda}$ as the space with basis $e_{\tau}$ labelled by Young tableaux $\tau \in \SYT(\lambda)$ and $\sH_n$ action given by:
    \begin{itemize}
        \item $T_i(e_{\tau}) = e_{\tau}$ if $i,i+1$ are in the same row of $\tau$
        \item $T_i(e_{\tau}) = -qe_{\tau}$ if $i,i+1$ are in the same column of $\tau$
        \item $T_i(e_{\tau}) = e_{s_i(\tau)} + \frac{(1-q)q^{c_{\tau}(i)}}{q^{c_{\tau}(i)} - q^{c_{\tau}(i+1)}} e_{\tau}$ if $s_i(\tau) > \tau$
        \item $T_i(e_{\tau}) =  -\frac{(q^{c_{\tau}(i+1)+1}-q^{c_{\tau}(i)})(q^{c_{\tau}(i)+1}-q^{c_{\tau}(i+1)})}{(q^{c_{\tau}(i+1)}-q^{c_{\tau}(i)})^2} e_{s_i(\tau)} + \frac{(1-q)q^{c_{\tau}(i)}}{q^{c_{\tau}(i)} - q^{c_{\tau}(i+1)}} e_{\tau}$ if $s_i(\tau) < \tau.$
    \end{itemize}

For $n \geq n_{\lambda}$ define the $\sA_n$ modules $U_{\lambda}^{(n)}:= \rho_n^{*}(S_{\lambda^{(n)}})$ and maps $\kappa_{\lambda}^{(n)}: U_{\lambda}^{(n+1)} \rightarrow U_{\lambda}^{(n)}$ given for $\tau \in \SYT(\lambda^{(n+1)})$ as

$$\kappa_{\lambda}^{(n)}(e_{\tau}) := \begin{cases}
    e_{\tau|_{\lambda^{(n)}}} & \tau(\square_0) = n+1\\
    0 & \tau(\square_0) \neq n+1.
     \end{cases} $$

where $\square_0$ is the unique square in $\lambda^{(n+1)}/\lambda^{(n)}.$

We consider the $\sA_n$ modules $U_{\lambda}^{(n)}$ as graded with the trivial grading i.e. $U_{\lambda}^{(n)} = U_{\lambda}^{(n)}(0).$ We will write $C_{\lambda}$ for the family 
$$C_{\lambda}:= \left( (U_{\lambda}^{(n)})_{n \geq n_{\lambda}}, (\kappa_{\lambda}^{(n)})_{n\geq n_{\lambda}}\right).$$
\end{defn}

\begin{remark}
    It is a straightforward (but tedious) exercise to show that the $\sH_n$ module defined on $S_{\lambda}$ as above is actually a $\sH_n$ module. We leave this to the reader. In fact this representation is the unique irreducible $\sH_n$ module corresponding to the shape $\lambda$. Definition \ref{vv polynomial spaces} follows the conventions of the author's paper \cite{weising2023murnaghantype} which are slightly different than the definitions of the spaces of \textit{vector valued polynomials} considered by Dunkl-Luque in \cite{DL_2011}. As constructed, the elements $e_{\tau}$ of the $\sA_{n}$ module $U_{\lambda}^{(n)}$ are \textbf{not} weight vectors for the Cherednik elements $Y_i$ but rather for the reversed orientation Cherednik elements $\theta_i$ given by $\theta_i = q^{i-1}T_{i-1}^{-1}\cdots T_1^{-1} \pi T_{n-1}\cdots T_i.$ Explicitly, we have that for $\tau \in \SYT(\lambda)$ and $1 \leq i \leq n $, $\theta_i(e_{\tau}) = q^{c_{\tau}(i)} e_{\tau}.$
\end{remark}

\begin{thm}\label{Murnaghan-type reps}
    For any $\lambda,$ $\Ind(C_{\lambda})$ is a compatible sequence with $\fL_{\bullet}(\Ind(C_{\lambda}))$ a \textit{nonzero} graded $\mathbb{B}_{q,t}$ module.
\end{thm}
\begin{proof}
    It is easy from the explicit $T_i$ relations given in Definition \ref{vv polynomial spaces} to verify that for every $n \geq n_{\lambda}$ the map $\kappa_{\lambda}^{(n)}:U_{\lambda}^{(n+1)} \rightarrow U_{\lambda}^{(n)}$ is a $\sH_n$ module map. We therefore also have that 
    \begin{align*}
        & \kappa_{\lambda}^{(n)}\pi^{(n+1)}T_n \\
        &= \kappa_{\lambda}^{(n)}\rho_{n+1}(\pi^{(n+1)})T_n\\
        &= \kappa_{\lambda}^{(n)}T_1^{-1}\cdots T_{n}^{-1}T_n\\
        &= \kappa_{\lambda}^{(n)}T_1^{-1}\cdots T_{n-1}^{-1}\\
        &= T_1^{-1}\cdots T_{n-1}^{-1}\kappa_{\lambda}^{(n)} \\
        &= \pi^{(n)}\kappa_{\lambda}^{(n)}.\\
    \end{align*}
    Hence, $C_{\lambda}$ is a pre-compatible sequence and so by Proposition \ref{pre-compatible implies compatible} it follows that $\Ind(C_{\lambda})$ is a compatible sequence. Thus we may consider the graded $\mathbb{B}_{q,t}$ module $\fL_{\bullet}(\Ind(C_{\lambda})).$

    To show that $\fL_{\bullet}(\Ind(C_{\lambda}))$ is nonzero it suffices to show that $\fL_{0}(\Ind(C_{\lambda}))$ is nonzero. This is space is the stable-limit of the symmetrized spaces $\epsilon_0^{(n)}(\Ind_{\sA_n}^{\sD_n^{+}} U_{\lambda}^{(n)})$ with respect to the maps $\kappa_{\lambda}^{(n)}.$ However, this space is the \textit{Murnaghan-type representation} $\widetilde{W}_{\lambda}$ of the positive elliptic Hall algebra of shape $\lambda$ from the author's prior paper \cite{weising2023murnaghantype}. This space is infinite dimensional for any $\lambda$ and so clearly $\fL_{0}(\Ind(C_{\lambda}))$ is nonzero.
\end{proof}

We can show further that for all $k\geq 0,$ $\fL_{k}(\Ind(C_{\lambda}))$ is infinite dimensional. To see this note that $d_{+}^{k}: \fL_{0}(\Ind(C_{\lambda})) \rightarrow \fL_{k}(\Ind(C_{\lambda}))$ is given by 
$$ (q^{k-1}X_1T_1^{-1}\cdots T_{k-1}^{-1})\cdots (q^2X_1T_1^{-1}T_2^{-1})(qX_1T_1^{-1})(X_1)$$
which is clearly injective. Thus since $\fL_{0}(\Ind(C_{\lambda})) = \widetilde{W}_{\lambda}$ is infinite-dimensional the same is true for $\fL_{k}(\Ind(C_{\lambda}))$.

As the $\mathbb{B}_{q,t}$ modules $\fL_{\bullet}(\Ind(C_{\lambda}))$ contain the Murnaghan-type representation $\widetilde{W}_{\lambda}$ of EHA we will refer to these modules as the $\mathbb{B}_{q,t}$ modules of Murnaghan-type.

\begin{remark}
    The author conjectures that each of the Murnaghan-type $\mathbb{B}_{q,t}$ modules, $\fL_{\bullet}(\Ind(C_{\lambda}))$, has an extended action by $\mathbb{B}_{q,t}^{\ext}$ and that these extended modules are calibrated. Evidence for this conjecture comes from the author's prior paper \cite{weising2023murnaghantype} where the author constructs $\Delta$-operators on the space $\widetilde{W}_{\lambda} = \fL_{0}(\Ind(C_{\lambda}))$ which have distinct nonzero spectrum. Extending these $\Delta$-operators to the whole space $\fL_{\bullet}(\Ind(C_{\lambda}))$ is nontrivial. The author will investigate this in a subsequent paper.
\end{remark}

\printbibliography

\end{document}